\documentclass[reqno,12pt]{amsart}
\usepackage{amsmath,amsthm,amssymb,amsfonts,amscd}
\usepackage{epsfig}
\usepackage[all]{xy}
\numberwithin{equation}{section}
\theoremstyle{plain}
\newtheorem{theorem}{Theorem}
\newtheorem{conjecture}[theorem]{Conjecture}
\newtheorem{lemma}[theorem]{Lemma}
\newtheorem{corollary}[theorem]{Corollary}
\newtheorem{proposition}[theorem]{Proposition}
\newtheorem*{theorem*}{Theorem}
\newtheorem*{conjecture*}{Conjecture}

\theoremstyle{definition}
\newtheorem{remark}[theorem]{Remark}

\newcommand{\CC}{{\mathbb{C}}}

\newcommand{\QQ}{{\mathbb{Q}}}

\newcommand{\RR}{{\mathbb{R}}}
\newcommand{\ZZ}{{\mathbb{Z}}}

\def\p{\partial }

\begin{document}
\title{A geometric definition of Gabrielov numbers}
%\date{\today}
\author{Wolfgang Ebeling and Atsushi Takahashi}
\address{Institut f\"ur Algebraische Geometrie, Leibniz Universit\"at Hannover, Postfach 6009, D-30060 Hannover, Germany}
\email{ebeling@math.uni-hannover.de}
\address{
Department of Mathematics, Graduate School of Science, Osaka University, 
Toyonaka Osaka, 560-0043, Japan}
\email{takahashi@math.sci.osaka-u.ac.jp}
%\date{\today}
\begin{abstract} Gabrielov numbers describe certain Coxeter-Dynkin diagrams of the 14 exceptional unimodal singularities and play a role in Arnold's strange duality.
In a previous paper, the authors defined Gabrielov numbers of a cusp singularity with an action of a finite abelian subgroup $G$ of ${\rm SL}(3,\CC)$ using the Gabrielov numbers of the cusp singularity and data of the group $G$. Here we consider a crepant resolution $Y \to \CC^3/G$ and the preimage  $Z$ of the image of the Milnor fibre of the cusp singularity under the natural projection $\CC^3 \to \CC^3/G$. Using the McKay correspondence, we compute the homology of the pair $(Y,Z)$. We construct a basis of the relative homology group $H_3(Y,Z;\QQ)$ with a Coxeter-Dynkin diagram where one can read off the Gabrielov numbers.
\end{abstract}
\maketitle
%%%%%%%%%%%%%%%%%%%%%%%%%%%%%%%%%%%%%%%%%%%%%%%%%%%%%%%%%%%%%%%%
\section*{Introduction}
V.~I.~Arnold \cite{Ar} observed a strange duality between the 14 exceptional unimodal singularities. For these singularities, two triples of numbers are defined. On one hand, by I.~V.~Dolgachev \cite{Do} the 14 exceptional unimodal singularities are related to triangles in the hyperbolic plane. If a singularity $X$ corresponds to a triangle with angles $\frac{\pi}{\alpha_1}$, $\frac{\pi}{\alpha_2}$, $\frac{\pi}{\alpha_3}$, then the {\em Dolgachev numbers} of $X$ are the numbers $\alpha_1, \alpha_2, \alpha_3$. On the other hand, A.~M.~Gabrielov \cite{Ga} has shown that such a singularity has a weakly distinguished basis of vanishing cycles with a Coxeter-Dynkin diagram being an extension of a diagram in the shape of the letter T. Namely, it is the extension of the diagram in Fig.~\ref{FigTpqr+} by an additional vertex $\delta_{\mu'+1}$ connected to the vertex $\delta_{\mu'}$ and to no other vertices. The lengths of the three arms $\gamma'_1, \gamma'_2, \gamma'_3$ are called the 
{\em Gabrielov numbers} of the singularity. Arnold observed that there is an involution $X \to X^\vee$ on the set of the 14 singularities such that the Dolgachev numbers of $X$ are the Gabrielov numbers of $X^\vee$ and the Gabrielov numbers of $X$ are the Dolgachev numbers of $X^\vee$.

By setting the module in the 14 singularity classes equal to zero, one finds quasi-homogeneous polynomials which can be given by invertible polynomials in three variables, i.e.\ polynomials which have as many monomials as they have variables. The authors have shown that there is a strange duality between non-degenerate invertible polynomials $f(x_1,x_2,x_3)$ in three variables \cite{ET1}. Namely, they defined Dolgachev and Gabrielov numbers for them and showed that the Berglund-H\"ubsch duality between these polynomials generalizes Arnold's strange duality. The Gabrielov numbers are defined by considering the deformation $f(x_1,x_2,x_3)-x_1x_2x_3$ which is right equivalent to a cusp singularity $x_1^{\gamma'_1} + x_2^{\gamma'_2} + x_3^{\gamma'_3} - cx_1x_2x_3$ for some $c \in \RR$, $c  \gg 0$. It has a Coxeter-Dynkin diagram as in Fig.~\ref{FigTpqr+}.

In later work \cite{ET}, the authors considered pairs $(f,G)$ where $f$ is an invertible polynomial in three variables and $G$ is a finite abelian group of symmetries of $f$ contained in ${\rm SL}(3,\CC)$. There is a dual pair $(f^T, G^T)$ where $f^T$ is the Berglund-H\"ubsch transpose and $G^T$ the dual group defined by Berglund and Henningson. We defined Gabrielov numbers for the pairs $(f,G)$ and Dolgachev numbers for the dual pairs $(f^T,G^T)$ and showed that they are the same. The Gabrielov numbers are defined in terms of the integers  $\gamma'_1, \gamma'_2, \gamma'_3$ associated to $f$ as above and data of the group $G$. Namely, for $i=1,2,3$, let $K_i$ be the maximal subgroup of $G$ fixing the $i$-th coordinate $x_i$, whose 
order $|K_i|$ is denoted by $n_i$. Then the Gabrielov numbers are the numbers $\frac{\gamma'_i}{|G/K_i|}$ repeated $n_i$ times ($i=1,2,3$) where numbers equal to one have to be omitted.

The objective of the present paper is to show that the Gabrielov numbers can be defined in a similar way as in Gabrielov's original definition, namely as lengths of arms in a certain Coxeter-Dynkin diagram. Roughly speaking, we consider a crepant resolution $h: Y \to \CC^3/G$ and the preimage  $Z=h^{-1}(W)$ of the image $W$ of the Milnor fibre $V$ of the cusp singularity under the natural projection $\CC^3 \to \CC^3/G$.  We compute the cohomology of the pair $(Y,Z)$ using the McKay correspondence for finite abelian subgroups of ${\rm SL}(3,\CC)$ \cite{IR}. We construct a basis of $H_3(Y,Z;\QQ)$ such that the Coxeter-Dynkin diagram contains a star-shaped graph with arms of lengths equal to the Gabrielov numbers and all but the central vertex corresponding to cycles with self-intersection number $-2$ (cf.\ Fig.~\ref{FighatTpqr}).
\begin{sloppypar}

{\bf Acknowledgements}.\  
This work has been supported 
by the DFG-programme SPP1388 ''Representation Theory'' (Eb 102/6-1).
The second named author is also supported 
by JSPS KAKENHI Grant Number 24684005. 
\end{sloppypar}
%%%%%%%%%%%%%%%%%%%%%%%%%%%%%%%%%%%%%%%%%%%%%%%%%%%%%%%%%%%%%%%%
\section{The McKay correspondence for finite abelian subgroups of ${\rm SL}(3,\CC)$}
Let $(\gamma'_1,\gamma'_2,\gamma'_3)$ be a triple of positive integers such that 
\begin{equation}\label{1.1}
\Delta(\gamma'_1,\gamma'_2,\gamma'_3):= \gamma'_1\gamma'_2\gamma'_3- \gamma'_2 \gamma'_3-\gamma'_1 \gamma'_3 - \gamma'_1 \gamma'_2>0.
\end{equation}
We associate to $(\gamma'_1,\gamma'_2,\gamma'_3)$ a polynomial 
$f(x_1,x_2,x_3):= x_1^{\gamma'_1} + x_2^{\gamma'_2} + x_3^{\gamma'_3} - cx_1x_2x_3$ for $c \in \RR$, $c  \gg 0$. 
Let $B_\varepsilon(0) \subset \CC^3$ be an open ball of sufficiently small radius $\varepsilon >0$ around the origin and 
let $V:= f^{-1}(\eta) \cap B_\varepsilon(0)$ be the Milnor fibre for sufficiently small $\eta \in \CC$, $0 < \eta \ll \varepsilon$.
Set 
\begin{equation}
\mu':=2+\sum_{i=1}^3\left(\gamma'_i-1\right).
\end{equation}
It is well-known that $\mu'$ is the Milnor number of the cusp singularity $V$.
Let $G$ be a finite abelian subgroup of ${\rm SL}(3,\CC)$ acting diagonally on $\CC^3$, 
and hence a subgroup of ${\rm SU}(3,\CC)$, such that $f$ is invariant under its natural action. 
Denote by $W$ the subvariety of $X:=B_\varepsilon(0)/G$ defined by the image of the smooth analytic subspace 
$V$ of $B_\varepsilon(0)$ under the natural projection $\pi:B_\varepsilon(0)\longrightarrow X$. 
Define a complex manifold $Y$ and a holomorphic map $h:Y\longrightarrow X$ by the following fibre product$:$
\begin{equation}
\xymatrix{
Y \ar[r]\ar[d]_h & \widetilde{\CC^3/G}\ar[d]\\
X=B_\varepsilon(0)/G\ar[r] & \CC^3/G
},
\end{equation}
where $\widetilde{\CC^3/G}$ denotes a crepant resolution of $\CC^3/G$ (e.g. $G\text{-}{\rm Hilb}(\CC^3)$). 
Then the map $h:Y\longrightarrow X$ gives a crepant resolution of the quotient $X$.
Set $Z:=h^{-1}(W)$ and denote the natural embedding $Z\hookrightarrow Y$ by $\iota$.
From the following relative homology long exact sequence
\[
\xymatrix{
\ar[r] & H_{q+1}(Y,Z;\QQ) \ar[r] & H_q(Z;\QQ)\ar[r]^{\iota_*} & H_q(Y;\QQ) \ar[r] & H_{q}(Y,Z;\QQ) \ar[r] &
}
\]
we get 
\begin{subequations}
\begin{equation}
H_p(Y,Z;\QQ)=0,\quad q \ne 2,3,4,
\end{equation}
\begin{equation}
H_4(Y,Z;\QQ)\cong H_4(Y;\QQ),
\end{equation}
and the exact sequence
\begin{equation}
\xymatrix{
0\ar[r] & H_3(Y,Z;\QQ) \ar[r] & H_2(Z;\QQ)\ar[r]^{\iota_\ast} & H_2(Y;\QQ) \ar[r] & H_2(Y,Z;\QQ) \ar[r] & 0.
}
\end{equation}
\end{subequations}
By this exact sequence, we shall always identify $H_3(Y,Z;\QQ)$ with its image in $H_2(Z;\QQ)$.
Similarly, from the relative cohomology long exact sequence we have 
\begin{subequations}
\begin{equation}
H^p(Y,Z;\QQ)=0,\quad q \ne 2,3,4,
\end{equation}
\begin{equation}
H^4(Y,Z;\QQ)\cong H^4(Y;\QQ),
\end{equation}
and the exact sequence
\begin{equation}
\xymatrix{
0\ar[r] & H^2(Y,Z;\QQ) \ar[r] & H^2(Y;\QQ)\ar[r]^{\iota^\ast} & H^2(Z;\QQ) \ar[r] & H^3(Y,Z;\QQ) \ar[r] & 0.
}
\end{equation}
\end{subequations}
Before describing the cohomology groups more in detail, we prepare some terminologies.
For each $g \in G$, denote by $N_g$ the dimension of the fixed locus which is a linear subspace of $\CC^3$.
Each element $g\in G$ has a unique expression of the form
\begin{equation}
g={\rm diag}(e^{\frac{2\pi a_1}{r}}, e^{\frac{2\pi a_2}{r}},e^{\frac{2\pi a_3}{r}}) \quad \mbox{with } 0 \leq a_i < r,
\end{equation}
where $r$ is the order of $g$.
The age of $g$, which is introduced in \cite{IR}, is defined as the rational number 
\begin{equation}
{\rm age}(g) := \frac{1}{r}\sum_{i=1}^3 a_i. 
\end{equation}
Since we assume that $G\subset {\rm SL}_3(\CC)$, then this number is an integer. 
For $i=1,2,3$, let $K_i$ be the maximal subgroup of $G$ fixing the $i$-th coordinate $x_i$, whose 
order $|K_i|$ is denoted by $n_i$.
\begin{proposition} \label{Prop|G|}
We have
\begin{equation}
\left|G\right|=1+2j_G+\sum_{i=1}^3\left(n_i-1\right),
\end{equation}
where $j_G$ is the number of elements  $g \in G$ such that $age(g)=1$ and $N_g=0$.
\end{proposition}
\begin{proof}
This follows from the classification of elements of $G$ according to their ages and dimensions of fixed loci. 
\end{proof}
\begin{proposition}[Ito--Reid, Corollary~1.5 and Theorem~1.6 in \cite{IR}]
We have 
\begin{subequations}
\begin{equation}
\dim_\QQ H_2(Y;\QQ)=\dim_\QQ H^2(Y;\QQ)=\#\{g\in G~\vert~ age(g)=1\}=j_G+\sum_{i=1}^3\left(n_i-1\right),
\end{equation}
\begin{equation}
\dim_\QQ H_4(Y;\QQ)= \dim_\QQ H^2_c(Y;\QQ)=\#\{g\in G~\vert~ age(g)=1, N_g=0\}=j_G.
\end{equation}
\end{subequations}
In particular, the natural map
\begin{equation}
H^2_c(Y;\QQ)\longrightarrow H^2(Y;\QQ)
\end{equation}
is injective.
\end{proposition}
For $i=1,2,3$, set
\begin{equation}
\gamma_i:=\frac{\gamma'_i}{\left|G/K_i\right|}.
\end{equation}
Since the analytic subspace $W$ of $X$ intersects the locus $\{x_2=x_3=0\}/G$ exactly at $\gamma_1$ points in 
$X\backslash\{0\}$, it turns out that $W$ has $A_{n_i-1}$-singularities at $\gamma_i$ distinct points for $i=1,2,3$ 
and it is smooth except for these $(\gamma_1+\gamma_2+\gamma_3)$ points. 
Therefore, we have the following.
\begin{proposition} \label{prop:H2Z}
Let 
\[ [E^i_{j,k}], \quad  i=1,2,3; \ j=1, \ldots, \gamma_i; \ k=1, \ldots n_i-1, \]
be the irreducible components of the exceptional divisors in $Z$
whose dual graph for fixed $i=1,2,3$ and $j=1, \ldots, \gamma_i$ is given by Fig.~\ref{FigAn}.  
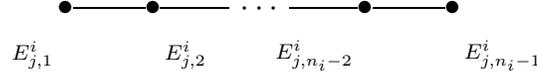
\begin{figure}
\[
\xymatrix{ 
  *{\bullet} \ar@{-}[r] \ar@{}_{E^i_{j,1}}[d]  &  *{\bullet} \ar@{-}[r] \ar@{}^{E^i_{j,2}}[d]  &  {\cdots} \ar@{-}[r]  & *{\bullet} \ar@{-}[r]    \ar@{}_{E^i_{j,n_i-2}}[d]  &*{\bullet} \ar@{}^{E^i_{j,n_i-1}}[d]   \\
  & & & &
  }
\]
\caption{Dual graph of $[E^i_{j,k}]$ for fixed $i$ and $j$} \label{FigAn}
\end{figure}
We have
\begin{equation}
H^2(Z;\QQ)\cong H^2(W;\QQ) \oplus
\left(\bigoplus_{i=1}^3\bigoplus_{j=1}^{\gamma_i}\bigoplus_{k=1}^{n_i-1} \QQ [E^i_{j,k}]^*\right),
\end{equation}
where $\{[E^i_{j,k}]^*\}$ is the dual basis of $\{[E^i_{j,k}]\}$.
\end{proposition}

As usual, in all the figures, $\bullet$ refers to an element of self-intersection number $-2$ and two vertices are connected by an edge if they correspond to elements with intersection number one.
\begin{corollary}\label{cor4}
The kernel of the map $\iota^\ast:H^2(Y;\QQ)\longrightarrow H^2(Z;\QQ)$ is $H^2_c(Y;\QQ)$, namely, we have 
\begin{equation}
H^2(Y,Z;\QQ)\cong H^2_c(Y;\QQ).
\end{equation}
Dually, the image of the map $\iota_\ast:H_2(Z;\QQ)\longrightarrow H_2(Y;\QQ)$ is given by
\begin{equation}
\bigoplus_{i=1}^3\bigoplus_{k=1}^{n_i-1}\QQ [E^i_{k}],
\end{equation}
where, for fixed $i=1,2,3$, $\{E^i_{k}\}$ are irreducible components of the exceptional divisors in $Y$ corresponding to 
elements of $K_i\setminus\{{\rm id}_G\}$ whose dual graph is given by Fig.~\ref{Fig:An}.  
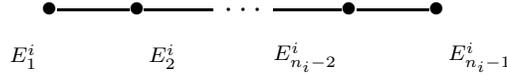
\begin{figure}
\[
\xymatrix{ 
  *{\bullet} \ar@{-}[r] \ar@{}_{E^i_{1}}[d]  &  *{\bullet} \ar@{-}[r] \ar@{}^{E^i_{2}}[d]  &  {\cdots} \ar@{-}[r]  & *{\bullet} \ar@{-}[r]    \ar@{}_{E^i_{n_i-2}}[d]  &*{\bullet} \ar@{}^{E^i_{n_i-1}}[d]   \\
  & & & &
  }
\]
\caption{Dual graph of $E^i_{k}$ for fixed $i$} \label{Fig:An}
\end{figure}
\end{corollary}
\begin{proof}
We have the decomposition 
\[
H^2(Y;\QQ)=H^2_c(Y;\QQ)\oplus \bigoplus_{i=1}^3\bigoplus_{k=1}^{n_i-1}\QQ [E^i_{k}]^\ast,
\]
where $\{[E^i_{k}]^\ast\}$ is the dual basis of $\{[E^i_{k}]\}$.
It follows from the construction of the cycle $[E^i_{j,k}]$ that for fixed $i, k$ we have $\iota_\ast([E^i_{j,k}])=[E^i_k]$ for all $j=1,\dots, \gamma_i$ and hence
$\iota^\ast([E^i_{k}]^\ast)=\displaystyle\sum_{j=1}^{\gamma_i}[E^i_{j,k}]^\ast$, which implies that ${\rm Ker}(\iota^\ast)\subset H^2_c(Y;\QQ)$. On the other hand, by the following commutative diagram
\[
\xymatrix{
H^2_c(Y;\QQ)\ar[r]\ar[d]_{\iota^\ast} & H^2(Y;\QQ)\ar[d]^{\iota^\ast}\\
H^2_c(Z;\QQ)\ar[r]& H^2(Z;\QQ)\\
},
\]
and by Poincar\'{e} duality $H^2_c(Z;\QQ)\cong H_4(Z;\QQ)^\ast=0$, we see that $H^2_c(Y;\QQ)\subset {\rm Ker}(\iota^\ast)$.
\end{proof}
In particular, we have a natural isomorphism 
\begin{equation}
H_4(Y,Z;\QQ)\cong H_4(Y;\QQ)\cong H^2_c(Y;\QQ)\cong H^2(Y,Z;\QQ)\cong H_2(Y,Z;\QQ)^\ast
\end{equation}
induced by Poincar\'{e} duality on $Y$. Therefore, we obtain
\begin{equation}
\dim_\QQ H_2(Y,Z;\QQ)=\dim_\QQ H_4(Y,Z;\QQ)=j_G.
\end{equation}
\begin{corollary}
We have 
\begin{equation}\label{eq:rank}
\dim_\QQ H_3(Y,Z;\QQ)=2+\sum_{i=1}^3n_i(\gamma_i-1).
\end{equation}
\end{corollary}
\begin{proof}
Note that $\dim_\QQ H^2(W;\QQ)=\dim_\QQ H^2(V;\QQ)^G$ where $H^2(V;\QQ)^G$ denotes 
the $G$-invariant subspace of $H^2(V;\QQ)$.
Since $H^2(V;\CC)^G$ is isomorphic to the $G$-invariant subspace of the Jacobian ring 
$\CC\{x_1,x_2,x_3\}/(\frac{\p f}{\p x_1},\frac{\p f}{\p x_2}, \frac{\p f}{\p x_3})$ of $f$, we have 
$\dim_\QQ H^2(V;\QQ)^G=2+\displaystyle\sum_{i=1}^{3}(\gamma_i-1)$.
Note also that the number of elements $g\in G$ such that ${\rm age}(g)=1$ and $N_g=1$ is given by $\displaystyle\sum_{i=1}^{3}(n_i-1)$.
The equality \eqref{eq:rank} follows.
\end{proof}
\begin{remark}
Denote by $\CC^{N_g}$ the fixed locus of $g\in G$. 
For $p\in\ZZ$, set 
\begin{equation}
H^p_G(B_\varepsilon(0),V;\QQ):=\bigoplus_{g\in G}{H}^{p+2 age(g)}(B_\varepsilon(0)\cap \CC^{N_g},V\cap \CC^{N_g};\QQ)^G,
\end{equation}
\begin{equation}
\widetilde{H}^{p-1}_G(V;\QQ):=\bigoplus_{g\in G}\widetilde{H}^{p+2 age(g)-1}(V\cap \CC^{N_g};\QQ)^G,
\end{equation}
where $(-)^G$ denotes the $G$-invariant subspace and $\widetilde{H}$ denotes the reduced cohomology.
Then by Example~5.9 in \cite{ET} we have isomorphisms of $\QQ$-vector spaces 
\begin{equation}
H^p(Y,Z;\QQ)\cong H^p_G(B_\varepsilon(0),V;\QQ)\cong \widetilde{H}^{p-1}_G(V;\QQ),\quad p\in\ZZ.
\end{equation}
\end{remark}
%%%%%%%%%%%%%%%%%%%%%%%%%%%%%%%%%%%%%
%%%%%%%%
\section{Coxeter-Dynkin diagrams}
Let $\langle -, -\rangle_V:H_2(V;\ZZ)\times H_2(V;\ZZ)\longrightarrow \ZZ$ be the intersection form. 
By \cite{Ga} the singularity $f(x,y,z)$ has a distinguished basis of vanishing cycles 
\[ \{ \delta_1; \delta_1^1, \delta_2^1, \ldots, \delta_{\gamma'_1-1}^1; \delta_1^2, \delta_2^2, \ldots, \delta_{\gamma'_2-1}^2; \delta_1^3, \delta_2^3, \ldots, \delta_{\gamma'_3-1}^3;\delta_{\mu'} \} \]
with a Coxeter-Dynkin diagram shown in Fig.~\ref{FigTpqr+}. Here the double dashed edge means intersection number $-2$.
\begin{figure}
$$
\xymatrix{ 
 & & & *{\bullet} \ar@{==}[d] \ar@{-}[dr]  \ar@{-}[ldd] \ar@{}^{\delta_{\mu'}}[r] 
 & & &  \\
 *{\bullet} \ar@{-}[r] \ar@{}_{\delta^2_{\gamma'_2-1}}[d]  & {\cdots} \ar@{-}[r]  & *{\bullet} \ar@{-}[r] \ar@{-}[ur]   \ar@{}_{\delta^2_1}[d] & *{\bullet} \ar@{-}[dl] \ar@{-}[r] \ar@{}^{\delta_1}[d] & *{\bullet} \ar@{-}[r]  \ar@{}^{\delta^3_1}[d]  & {\cdots} \ar@{-}[r]  &*{\bullet} \ar@{}^{\delta^3_{\gamma'_3-1}}[d]   \\
& &  *{\bullet} \ar@{-}[dl] \ar@{}_{\delta^1_1}[r]  & & & &  \\
 & {\cdots} \ar@{-}[dl] & & & & & \\
*{\bullet}  \ar@{}_{\delta^1_{\gamma^*_1-1}}[r] & & & & & &
  }
$$
\caption{Coxeter-Dynkin diagram of $f(x_1,x_2,x_3):= x_1^{\gamma'_1} + x_2^{\gamma'_2} + x_3^{\gamma'_3} - cx_1x_2x_3$} \label{FigTpqr+}
\end{figure}
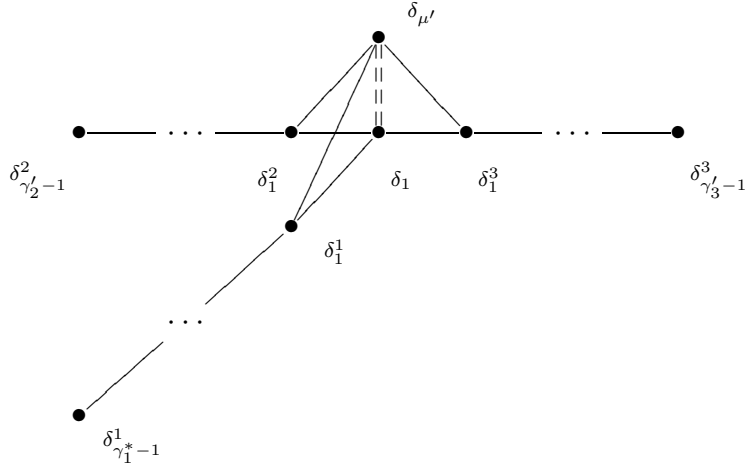

The natural map
\begin{equation}\label{Leray isom}
H_2(V;\ZZ)\longrightarrow H_3(B_\varepsilon(0) \setminus V; \ZZ)
\end{equation}
is an isomorphism since we have $H_3(B_\varepsilon(0);\ZZ)=H_4(B_\varepsilon(0);\ZZ)=0$.
Consider
\[ T^3 = \{ x \in B_\varepsilon(0) \, | \, |x_1|=|x_2|=|x_3|=\eta \} \subset B_\varepsilon(0). \]
Let $\delta_0 \in H_2(V, \ZZ)$ be the image of the class $[T^3] \in H_3(B_\varepsilon(0) \setminus V; \ZZ)$ under 
the isomorphism~\eqref{Leray isom}.
The radical of the lattice $H_2(V; \ZZ)$ is of rank one because of the condition~\eqref{1.1} 
and is generated by $\delta_0$ since $\delta_0$ is invariant under the Milnor monodromy.
We may assume that $\delta_0 = \delta_{\mu'}-\delta_1$ by redefinition of $\delta_{\mu'}$ if necessary. 
The lattice $H_2(V; \ZZ)/\langle \delta_0 \rangle$ has a basis of vanishing cycles
\[ {\mathcal B} := \{ \delta_1; \delta_1^1, \delta_2^1, \ldots, \delta_{\gamma'_1-1}^1; \delta_1^2, \delta_2^2, \ldots, \delta_{\gamma'_2-1}^2; \delta_1^3, \delta_2^3, \ldots, \delta_{\gamma'_3-1}^3 \} \]
intersecting with a Coxeter-Dynkin diagram shown in Fig.~\ref{FigTpqr}.

\begin{figure}
$$
\xymatrix{ 
 *{\bullet}  \ar@{-}[r] \ar@{}_{\delta_{\gamma'_2-1}^2}[d]  & {\cdots} \ar@{-}[r]  &  *{\bullet} \ar@{-}[r]    \ar@{}_{\delta_1^2}[d] &  *{\bullet} \ar@{-}[dl] \ar@{-}[r] \ar@{}_{\delta_1}[dr] &  *{\bullet} \ar@{-}[r]  \ar@{}^{\delta_1^3}[d]  & {\cdots} \ar@{-}[r]  & *{\bullet} \ar@{}^{\delta_{\gamma'_3-1}^3}[d]   \\
& &   *{\bullet} \ar@{-}[dl] \ar@{}_{\delta_1^1}[r]  & & & &  \\
 & {\cdots} \ar@{-}[dl] & & & & & \\
 *{\bullet}  \ar@{}_{\delta_{\gamma'_1-1}^1}[r] & & & & & &
  }
$$
\caption{Coxeter-Dynkin diagram corresponding to ${\mathcal B}$}  \label{FigTpqr}
\end{figure}
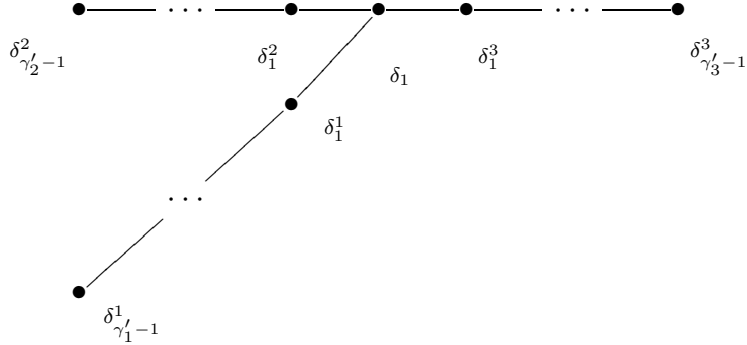

We consider the action of $G$ on $H_2(V; \ZZ)$.
Let $g = \left( \frac{a_1}{\gamma'_1}, \frac{a_2}{\gamma'_2}, \frac{a_3}{\gamma'_3} \right)$ be an element of $G$.  For $i=1,2,3$ and $j \in \ZZ$ define
\begin{eqnarray*} 
\delta_{\gamma'_i}^i & := & \delta_0 - \sum_{k=1}^{\gamma'_i-1} \delta_k^i, \\
\delta_j^i & := & \delta_{j'}^i \quad \mbox{if } j-j' \equiv 0 \, \mbox{mod} \, \gamma'_i.
\end{eqnarray*}

\begin{lemma} \label{lemGaction}
We have
\begin{itemize}
\item[{\rm (i)}] $g(\delta_0) = \delta_0$.
\item[{\rm (ii}]  $g(\delta_j^i) = \delta_{j+a_i}^i$ for $i=1,2,3$ and $j \in \ZZ$.
\item[{\rm (iii)}] $g(\delta_1) = \delta_1 + \sum_{i=1}^3 \sum_{j=1}^{a_i} \delta_j^i - \mbox{age}(g) \delta_0$.
\end{itemize}
\end{lemma} 

\begin{proof} (i) is clear by the construction.

(ii) For symmetry reasons, it suffices to prove the claim for $i=3$. We apply \cite[Theorem~1]{Ga} to the function $F_{\gamma_3'-1}(x,y,z)=f(x,y,z)$. By this theorem, the intersection matrix of $\{ \delta_1^3, \delta_2^3, \ldots, \delta_{\gamma'_3-1}^3 \}$ coincides with an intersection matrix for the singularity 
\[f'(x,y,z)  = x^{\gamma'_1} + y^{\gamma'_2} -c\beta xy+z^{\gamma'_3} \quad \mbox{for } \beta \neq 0.\]
This is a singularity of type $A_{\gamma'_3-1}$ and $\{ \delta_1^3, \delta_2^3, \ldots, \delta_{\gamma'_3-1}^3 \}$ is a distinguished basis of vanishing cycles for this singularity with the standard Coxeter-Dynkin diagram. It is easy to see that the action of $g$ on this basis is as claimed.

(iii) We set 
\[ g(\delta_1) = \xi_1 \delta_1 + \sum_{i=1}^3 \sum_{j=1}^{\gamma_i'-1} \xi^i_j \delta_j^i \]
and solve the equations
\[ \langle g(\delta_1), g(\delta_j^i) \rangle_V = \langle \delta_1 , \delta_j^i \rangle_V \]
for the unknown coefficients $\xi_1$, $\xi_j^i$ using (ii) to obtain (iii).
\end{proof}

Now we shall consider the homology group $H_2(W; \QQ)$. 
For $\delta\in H_2(V;\ZZ)$, we set $\overline{\delta}:=(\pi|_V)_*(\delta)$ where $\pi|_V$ is the natural projection $\pi|_V:V\longrightarrow W$.
Since Poincar\'{e} duality holds on $W$ over $\QQ$ since 
$W=V/G$ is a quotient of a smooth manifold $V$ by a finite group $G$
we have the intersection form $\langle -, -\rangle_W:H_2(W; \QQ)\times H_2(W; \QQ)\longrightarrow \QQ$, which satisfies 
\[ \langle \overline{\delta}, \overline{\delta}' \rangle_W = \sum_{g \in G} \langle \delta, g(\delta') \rangle_V,
\quad \delta,\delta'\in H_2(V;\ZZ). \]
Clearly $\overline{\delta}_0$ lies in the radical of $H_2(W; \QQ)$. 
Then $H_2(W; \QQ)/\langle \overline{\delta}_0 \rangle$ has a $\QQ$-basis $\overline{\mathcal B}$ where 
\[ \overline{\mathcal B} := \{\overline{\delta}_1;  \overline{\delta}_1^1, \overline{\delta}_2^1, \ldots, \overline{\delta}_{\gamma_1-1}^1; \overline{\delta}_1^2, \overline{\delta}_2^2, \ldots, \overline{\delta}_{\gamma_2-1}^2; \overline{\delta}_1^3, \overline{\delta}_2^3, \ldots, \overline{\delta}_{\gamma_3-1}^3 \}.\]

\begin{lemma} \label{lembar}
We have
\begin{eqnarray}
\langle \overline{\delta}_1 , \overline{\delta}_1 \rangle_W  & =  & 2 j_G -2, \label{eqbar1} \\
\langle \overline{\delta}_j^i , \overline{\delta}_j^i \rangle_W & = & -2 n_i  \mbox{ for } i=1,2,3 \mbox{ and } j=1, \ldots, \gamma_i -1,  \label{eqbar2} \\
\langle \overline{\delta}_1 , \overline{\delta}_1^i \rangle_W & =  & \langle \overline{\delta}_j^i , \overline{\delta}_{j+1}^i \rangle_W = n_i  \mbox{ for } i=1,2,3 \mbox{ and } j=1, \ldots, \gamma_i -2,
\label{eqbar3}
\end{eqnarray}
and $\langle \overline{\delta}_j^i , \overline{\delta}_{j'}^{i'} \rangle_W  =  0$ otherwise.
\end{lemma}

\begin{proof} This follows from the definition of $\langle -, - \rangle_W$ and Lemma~\ref{lemGaction}. We have 
\[ \langle \overline{\delta}_1 , \overline{\delta}_1 \rangle_W  = \langle \delta_1, \delta_1 \rangle_V |G| + \sum_{i=1}^3 n_i (|G/K_i|-1) = 2 j_G -2, \]
where the last equality follows from Proposition~\ref{Prop|G|}. This proves (\ref{eqbar1}). To prove (\ref{eqbar2}), note that $g(\delta_j^i) \neq \delta_{j+1}^i$ unless $|G/K_i|=\gamma_i'$ and hence $\gamma_i=1$. Therefore, if $\gamma_i > 1$ and $j=1, \ldots, \gamma_i -1$,
\[ \langle \overline{\delta}_j^i , \overline{\delta}_j^i \rangle_W = \sum_{g \in K_i} \langle \delta_j^i, \delta_j^i \rangle_V = -2 n_i. \]
Similarly one can prove (\ref{eqbar3}).
\end{proof}

The Coxeter-Dynkin diagram corresponding to $\overline{\mathcal B}$  is shown in Fig.~\ref{FigbarTpqr}. Here, in the circles representing the vertices, the self-intersection numbers of the corresponding cycles are indicated.

\begin{figure}
$$
\xymatrix{ 
 *+=[o][F-:<3pt>]\txt{\ $\scriptstyle{-2n_2}$\, \,}  \ar@{-}[r]^{n_2} \ar@{}_{\overline{\delta}_{\gamma_2-1}^2}[d]  & {\cdots} \ar@{-}[r]^{n_2}  &  *+=[o][F-:<3pt>]\txt{\ $\scriptstyle{-2n_2}$\,\, } \ar@{-}[r]^{n_2}    \ar@{}_{\overline{\delta}_1^2}[d] &  *+=[o][F-:<3pt>]\txt{\ $\scriptstyle{}2j_G-2$\,\, } \ar@{-}[dl]^{n_1} \ar@{-}[r]^{n_3} \ar@{}_{\overline{\delta}_1}[dr] &  *+=[o][F-:<3pt>]\txt{\ $\scriptstyle{-2n_3}$\,\, } \ar@{-}[r]^{n_3}  \ar@{}^{\overline{\delta}_1^3}[d]  & {\cdots} \ar@{-}[r]^{n_3}  & *+=[o][F-:<3pt>]\txt{\ $\scriptstyle{-2n_3}$\,\,} \ar@{}^{\overline{\delta}_{\gamma_3-1}^3}[d]   \\
& &   *+=[o][F-:<3pt>]\txt{\ $\scriptstyle{-2n_1}$\,\, } \ar@{-}[dl]^{n_1} \ar@{}_{\overline{\delta}_1^1}[r]  & & & &  \\
 & {\cdots} \ar@{-}[dl]^{n_1} & & & & & \\
 *+=[o][F-:<3pt>]\txt{\ $\scriptstyle{-2n_1}$\,\, }  \ar@{}_{\overline{\delta}_{\gamma_1-1}^1}[r] & & & & & &
  }
$$
\caption{Coxeter-Dynkin diagram corresponding to $\overline{\mathcal B}$}  \label{FigbarTpqr}
\end{figure}

We now want to consider the homology group $H_3(Y,Z; \QQ)$. 
Let $\langle -, -\rangle_Z:H_2(Z; \ZZ)\times H_2(Z; \ZZ)\longrightarrow \ZZ$ be the intersection form on $Z$. 
Consider the homology classes
\[[E^i_{j,k}]  \in H_2(Z; \ZZ), \quad  i=1,2,3; \ j=1, \ldots, \gamma_i; \ k=1, \ldots n_i-1, \]
of irreducible components of exceptional divisors defined in Proposition~\ref{prop:H2Z}. 
Fix $i=1,2,3$ with $n_i>1$ and $j=1, \ldots, \gamma_i$.
Then the  elements $[E^i_{j,1}], \ldots , [E^i_{j,n_i-1}]$ are simple roots of the root system $A_{n_i-1}$ (with the bilinear form multiplied by $-1$) with the standard Coxeter-Dynkin diagram shown in Fig~\ref{FigAn}. Let
\[ \Lambda^i_{j,1} := \frac{1}{n_i} \left( (n_i-1)[E^i_{j,1}] + (n_i-2) [E^i_{j,2}] + \cdots + [E^i_{j,n_i-1}] \right) \]
be the first fundamental weight of the root system $A_{n_i-1}$, which is an element of $H_2(Z; \QQ)$. 
Let $w^i_{j,k}$ denote the reflection corresponding to the root $[E^i_{j,k}]$. Define
\[ \lambda^i_{j,k} := w^i_{j,k} \cdots w^i_{j,1} (\Lambda^i_{j,1}) \quad \mbox{for } i=1,2,3, \ j=1, \ldots, \gamma_i, \
k=0, \ldots , n_i-1. \] 

\begin{lemma} \label{lemAn}
We have for $i=1,2,3$ with $n_i>1$  and $j=1, \ldots , \gamma_i$
\begin{eqnarray}
\langle \lambda^i_{j,k}, \lambda^i_{j,k} \rangle_Z  & = &  - \frac{n_i -1}{n_i} \quad \mbox{for } k=0, \ldots , n_i-1, \\
\langle \lambda^i_{j,k}, \lambda^i_{j,k'} \rangle_Z  & = &  \frac{1}{n_i} \quad \mbox{for }  0 \leq k,k' \leq n_i-1, k \neq k'. \label{eqkk'}
\end{eqnarray}
\end{lemma}

\begin{proof} We fix $i \in \{1,2,3\}$ with $n_i>1$ and $j \in \{1, \ldots, \gamma_i-1 \}$. In order to simplify notation, we set
\[ \Lambda_1 : = \Lambda^i_{j,1}, \quad  w_k := w^i_{j,k}, \quad \lambda_k := \lambda^i_{j,k}. \]
We have for $k=0, \ldots , n_i-1$
\[ \langle \lambda_k , \lambda_k \rangle_Z = \langle \Lambda_1, \Lambda_1 \rangle_Z = - \frac{n_i -1}{n_i}. \]
Since $w_k$ and $w_{k'}$ commute if $|k-k'|>1$ and $w_k(\Lambda_1)=\Lambda_1$ for $k=2, \ldots n_i-1$, it suffices to prove Equation~(\ref{eqkk'}) for $k'=k+1$. We prove this by induction on $k$. We have
\[ \langle \lambda_1, \lambda_2 \rangle_Z = \langle w_1(\Lambda_1), w_2w_1(\Lambda_1) \rangle_Z= \langle \Lambda_1, w_1w_2w_1(\Lambda_1) \rangle_Z = \langle \Lambda_1, w_2w_1w_2(\Lambda_1) \rangle_Z = \langle \Lambda_1, w_1(\Lambda_1) \rangle_Z \]
by the relations in the Weyl group and
\[ \langle \Lambda_1, w_1(\Lambda_1) \rangle_Z = \langle \Lambda_1, \Lambda_1 + e_1 \rangle_Z =  - \frac{n_i -1}{n_i} +1 = \frac{1}{n_i}. \]
Now let $k=2, \ldots, n_i-2$. The induction step follows from
\begin{eqnarray*}
\langle \lambda_k, \lambda_{k+1} \rangle_Z & = & \langle w_k w_{k-1} \cdots w_1(\Lambda_1) , w_{k+1} w_k w_{k-1} \cdots w_1(\Lambda_1) \rangle_Z \\
 &=& \langle w_{k-1} \cdots w_1(\Lambda_1), w_k w_{k+1} w_k w_{k-1} \cdots w_1(\Lambda_1) \rangle_Z \\
 & = &  \langle w_{k-1} \cdots w_1(\Lambda_1), w_{k+1} w_k w_{k+1} w_{k-1} \cdots w_1(\Lambda_1) \rangle_Z \\
 & = & \langle \lambda_{k-1} , \lambda_k \rangle_Z.
\end{eqnarray*}
\end{proof}

Let $h^{!}: H_2(W; \QQ) \longrightarrow H_2(Z;\QQ)$ be the composition of maps 
\[
H_2(W; \QQ) \stackrel{\cong}{\longrightarrow } H^2_c(W;\QQ) \stackrel{(h|_Z)^\ast}{\longrightarrow} H^2_c(Z;\QQ)
\stackrel{\cong}{\longrightarrow } H_2(Z;\QQ).
\]
Note that the pull-back map $(h|_Z)^\ast$ is well-defined since the map $h|_Z : Z \to W$ is proper.
For $i=1,2,3$ and $j=1, \ldots, \gamma_i-1$, set 
\begin{eqnarray*}
\widehat{\delta}_0 & := & h^{!}\left(\frac{1}{|G|}\overline{\delta}_0\right),\\
\widehat{\delta}_1 & := & h^{!}(\overline{\delta}_1),\\
\widehat{\delta}^i_{j,0}  & := &  h^{!}(\overline{\delta}_j^i) \quad  \mbox{ if } n_i=1, \\
\widehat{\delta}^i_{j,k}  & := & h^{!}\left(\frac{1}{n_i} \overline{\delta}_j^i\right)  + (\lambda^i_{j,k} - \lambda^i_{j+1,k}) \quad\mbox{for }  k=0, \ldots, n_i-1 \mbox{ otherwise}. 
\end{eqnarray*}
The element $\widehat{\delta}_0$ is again in the radical of $H_2(Z; \QQ)$.
\begin{lemma} 
The elements $\widehat{\delta}_0$, $\widehat{\delta}_1$, $\widehat{\delta}^i_{j,k}$ $(i=1,2,3$, $j=1,\dots, \gamma_i-1$, $k=0,\dots, n_i-1)$ belong to $H_3(Y,Z;\QQ)$.
\end{lemma}
\begin{proof}
Recall that $H_3(Y,Z;\QQ)\cong {\rm Ker}(\iota_\ast)$. 
By Corollary~\ref{cor4}, we have $\iota_\ast(\widehat{\delta}_0)=\iota_\ast(\widehat{\delta}_1)=0$ and 
\[
\iota_\ast(\widehat{\delta}^i_{j,k} )=\iota_\ast(\lambda^i_{j,k} - \lambda^i_{j+1,k})=0,
\]
since $\iota_\ast[E^i_{j,k}] =[E^i_k]=\iota_\ast[E^i_{j+1,k}]$ for fixed $i,k$ and for all $j=1,\dots, \gamma_i-1$.
\end{proof}

\begin{lemma} \label{lemhat}
We have 
\begin{eqnarray}
\langle \widehat{\delta}_1 , \widehat{\delta}_1 \rangle_Z  & = &  2j_G-2, \label{eqhat11} \\
\langle \widehat{\delta}^i_{j,k} , \widehat{\delta}^i_{j,k} \rangle_Z & = & -2 \mbox{ for } i=1,2,3, j=1, \ldots , \gamma_i -1, k=0, \ldots, n_i-1, \label{eqhatijkijk} \\
\langle \widehat{\delta}_1 ,  \widehat{\delta}^i_{1,k} \rangle_Z & = & \langle \widehat{\delta}^i_{j,k}  \widehat{\delta}^i_{j+1,k} \rangle_Z = 1  \mbox{ for } i=1,2,3, j=1, \ldots , \gamma_i -2, k=0, \ldots, n_i-1,
\label{eqhatijk+1}
\end{eqnarray}
and $\langle \widehat{\delta}^i_{j,k} , \widehat{\delta}^{i'}_{j',k'} \rangle_Z=0$ otherwise.
\end{lemma}

\begin{proof} Equation~(\ref{eqhat11}) is obvious.

Let $i=1,2,3$. For $n_i=1$, Equation~(\ref{eqhatijkijk}) is obvious. For $n_i>1$, $j=1, \ldots , \gamma_i -1$ and $k=0, \ldots, n_i-1$, we have
\[ \langle \widehat{\delta}^i_{j,k} , \widehat{\delta}^i_{j,k} \rangle_Z = \frac{1}{n_i^2} \langle h^{!}(\overline{\delta}_j^i) , h^{!}(\overline{\delta}_j^i) \rangle_Z + \langle \lambda^i_{j,k}, \lambda^i_{j,k} \rangle_Z + \langle \lambda^i_{j+1,k} , \lambda^i_{j+1,k} \rangle_Z = \frac{-2n_i}{n_i^2} - 2 \frac{n_i-1}{n_i} = -2
\]
by Lemma~\ref{lembar} and Lemma~\ref{lemAn}.

It is clear that $\langle \widehat{\delta}^i_{j,k} , \widehat{\delta}^{i'}_{j',k'} \rangle_Z=0$ for $i \neq i'$ or $i=i'$ and $|j'-j|>1$. 
Let $n_i>1$, $j=1, \ldots , \gamma_i-1$, and $0 \leq k,k' \leq n_i-1$, $k \neq k'$. Then we have
\[ \langle \widehat{\delta}^i_{j,k}, \widehat{\delta}^i_{j,k'} \rangle_Z = \frac{1}{n_i^2} \langle h^{!}(\overline{\delta}_j^i) , h^{!}(\overline{\delta}_j^i) \rangle_Z + \langle \lambda^i_{j,k}, \lambda^i_{j,k'} \rangle_Z + \langle \lambda^i_{j+1,k} , \lambda^i_{j+1,k'} \rangle_Z = \frac{-2}{n_i} +2 \frac{1}{n_i} = 0,
\]
again by Lemma~\ref{lembar} and Lemma~\ref{lemAn}. Now let $j=1, \ldots, \gamma_i-2$ and $0 \leq k,k' \leq n_i-1$. Again, Equation~(\ref{eqhatijk+1}) is obvious for $n_i=1$. Therefore, let $n_i > 1$. First consider the case $k'=k$. Then we have
\[ \langle \widehat{\delta}^i_{j,k}, \widehat{\delta}^i_{j+1,k} \rangle_Z = \frac{1}{n_i^2} \langle h^{!}(\overline{\delta}_j^i) , h^{!}(\overline{\delta}_{j+1}^i) \rangle_Z + \langle - \lambda^i_{j+1,k} ,
\lambda^i_{j+1,k} \rangle_Z = \frac{1}{n_i} - \left( - \frac{n_i - 1}{n_i} \right) = 1. \]
Finally we have for $k' \neq k$
\[ \langle \widehat{\delta}^i_{j,k}, \widehat{\delta}^i_{j+1,k'} \rangle_Z = \frac{1}{n_i^2} \langle h^{!}(\overline{\delta}_j^i) , h^{!}(\overline{\delta}_{j+1}^i) \rangle_Z + \langle - \lambda^i_{j+1,k} ,
\lambda^i_{j+1,k'} \rangle_Z = \frac{1}{n_i} - \frac{1}{n_i} = 0. \]
\end{proof}

Set
\[ \widehat{\mathcal B} := \{ \widehat{\delta}_1 \} \cup \{ \widehat{\delta}^i_{j,k} \, | \,  i=1,2,3, \ j=1, \ldots , \gamma_i -1, \ k=0, \ldots, n_i-1 \}. \]

To summarize the results in this section, we have
\begin{theorem}
The set $\widehat{\mathcal B} \cup \{ \widehat{\delta}_0 \}$ is contained in $H_3(Y,Z; \QQ) \subset H_2(Z; \QQ)$ and  $\widehat{\mathcal B}$ represents a $\QQ$-basis of $H_3(Y,Z; \QQ)/\langle \widehat{\delta}_0 \rangle$ whose intersection numbers are given by the Coxeter-Dynkin diagram as in Fig.~\ref{FighatTpqr}.
\end{theorem}

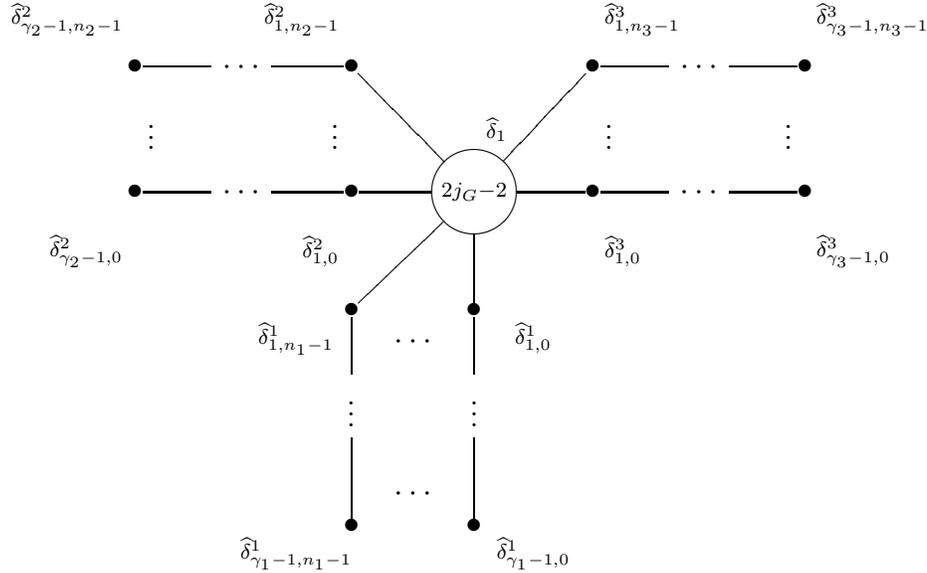
\begin{figure}
$$
\xymatrix{ & & & & & & \\
 *{\bullet}   \ar@{-}[r] \ar@{}^{\widehat{\delta}_{\gamma_2-1,n_2-1}^2}[u]  \ar@{}^{\vdots}[d] & {\cdots}  \ar@{-}[r]  &  *{\bullet} \ar@{-}[dr]    \ar@{}^{\widehat{\delta}_{1,n_2-1}^2}[u] \ar@{}_{\vdots}[d] &  &  *{\bullet} \ar@{-}[dl]  \ar@{-}[r]  \ar@{}_{\widehat{\delta}_{1,n_3-1}^3}[u] \ar@{}^{\vdots}[d] & {\cdots} \ar@{-}[r]  & *{\bullet} \ar@{}_{\widehat{\delta}_{\gamma_3-1,n_3-1}^3}[u]  \ar@{}_{\vdots}[d] \\
 *{\bullet}  \ar@{-}[r] \ar@{}_{\widehat{\delta}_{\gamma_2-1,0}^2}[d]  & {\cdots} \ar@{-}[r]  &  *{\bullet} \ar@{-}[r]    \ar@{}_{\widehat{\delta}_{1,0}^2}[d]  & *+=[o][F-:<3pt>]\txt{\ $\scriptstyle{2j_G-2}$\,\, } \ar@{-}[dl] \ar@{-}[d] \ar@{}_{\widehat{\delta}_1}[u]  & *{\bullet} \ar@{-}[l]  \ar@{-}[r]  \ar@{}^{\widehat{\delta}_{1,0}^3}[d]  & {\cdots} \ar@{-}[r]  & *{\bullet} \ar@{}^{\widehat{\delta}_{\gamma_3-1,0}^3}[d]  \\
& \ar@{}_{\widehat{\delta}_{1,n_1-1}^1}[r] &   *{\bullet} \ar@{-}[d]  \ar@{}_*+{\cdots}[r] & *{\bullet} \ar@{-}[d] \ar@{}_{\widehat{\delta}_{1,0}^1}[r] & & &  \\
 & & {\vdots} \ar@{-}[d] & {\vdots} \ar@{-}[d]  & & & \\
 & \ar@{}_{\widehat{\delta}_{\gamma_1-1,n_1-1}^1}[r] & *{\bullet}  \ar@{}^*+{\cdots}[r]   & *{\bullet}  \ar@{}_{\widehat{\delta}_{\gamma_1-1,0}^1}[r] & & &
  }
$$
\caption{Coxeter-Dynkin diagram corresponding to $\widehat{\mathcal B}$}  \label{FighatTpqr}
\end{figure}
We expect the following.
\begin{conjecture}
The set $\widehat{\mathcal B} \cup \{ \widehat{\delta}_0 \}$ represents a $\ZZ$-basis of $H_3(Y,Z; \ZZ)$.
\end{conjecture}

\begin{remark} The classes $\widehat{\delta}^i_{j,k}$ ($i=1,2,3$, $j=1, \ldots, \gamma_i-1$, $k=0, \ldots, n_i-1$) are vanishing classes since
$\lambda^i_{j,k} -\lambda^i_{j+1,k}$ tends to zero
when $\eta$ tends to zero.
\end{remark}

%%%%%%%%%%%%%%%%%%%%%%%%%%%%%%%%%%%%%%%

\end{document}